\documentclass{article}
\usepackage[utf8]{inputenc}
\usepackage{fullpage}
\usepackage{amsmath,amsfonts,amssymb,amsthm,enumerate,mdframed,hyperref,cleveref,algorithmic,algorithm,multirow,import,tikz,booktabs,xcolor,soul}
\usetikzlibrary{arrows,positioning, shapes} 
\usepackage{caption,subcaption}
\usepackage{enumitem}
\usepackage{authblk}

\newtheorem{theorem}{Theorem}
\newtheorem{proposition}[theorem]{Proposition}
\newtheorem{lemma}[theorem]{Lemma}

\newtheorem{definition}[theorem]{Definition}

\newtheorem{remark}[theorem]{Remark}

\mathcode`*=\numexpr\mathcode`*-"2000\relax

\definecolor{darkgreen}{rgb}{0,0.6,0}

\providecommand{\F}{\mathbb{F}}

\DeclareMathOperator{\PG}{PG}
\DeclareMathOperator{\supp}{supp}

\title{Divisible minimal codes}
\author{Vladimir Chubenko}
\affil{Amateur mathematician, Ukraine, chubenko.vl@gmail.com}
\author{Sascha Kurz}
\affil{Mathematisches Institut, Universität Bayreuth, D-95440 Bayreuth, Germany, sascha.kurz@uni-bayreuth.de}
\date{}

\begin{document}

\maketitle

\begin{abstract}
  Minimal codes are linear codes where all non-zero codewords are minimal, i.e., whose support is not properly contained in the support of another codeword. 
  The minimum possible length of such a $k$-dimensional linear code over $\mathbb{F}_q$ is denoted by $m(k,q)$. Here we determine $m(7,2)$, $m(8,2)$, and $m(9,2)$, 
  as well as full classifications of all codes attaining $m(k,2)$ for $k\le 7$ and those attaining $m(9,2)$. 
  We give improved upper bounds for $m(k,2)$ for all $10\le k\le 17$.  
  It turns out that in many cases the attaining extremal codes have the property that the weights of all codewords are divisible by some constant $\Delta>1$. So, here 
  we study the minimum lengths of minimal codes where we additionally assume that the weights of the codewords are divisible by $\Delta$. As a byproduct 
  we also give a few binary linear codes improving the best known lower bound for the minimum distance.
\end{abstract}

\section{Introduction}

Let $\mathbb{F}_q$ be a finite field of cardinality $q$ and $C\subseteq \mathbb{F}_q^n$ be a linear code. If $C$ has cardinality $q^k$, then 
we speak of an $[n,k]_q$-code. A non-zero codeword $c\in C$ is called \emph{minimal} if the \emph{support} $\supp(c):=\{i\mid c_i\neq 0\}$ of $c$ is minimal with respect to 
inclusion in the set $\left\{\supp(u) \mid u\in C\backslash\mathbf{0}\right\}$. The code $C$ is a \emph{minimal code} if all of its non-zero codewords are minimal. 
One of the many applications of minimal codes is secret sharing, see e.g.\ \cite{ashikhmin1998minimal}. An important line of research is the determination of the minimum possible length $n$ 
of a minimal $[n,k]_q$-code, which we denote by $m(k,q)$. In e.g.\ \cite[Theorem 2.14]{three_combinatorial_perspectives} the lower bound $m(k,q)\ge (q+1)(k-1)$ was shown. 
Here we determine $m(7,2)$, $m(8,2)$, and $m(9,2)$, as well as full classifications of all codes attaining $m(k,2)$ for $k\le 7$ 
and those attaining $m(9,2)$. For $m(k,2)$ we give improved upper bounds when $10\le k\le 17$. 

A linear $[n,k]_q$-code is called $\Delta$-divisible if all of its weights are divisible by $\Delta$. For some background we refer e.g.\ to the recent survey \cite{kurz2021divisible}. Minimal codes 
constructed by concatenation with simplex codes, see e.g.\ \cite{alfarano2023outer,bartoli2023small}, naturally come with a non-trivial divisibility constant $\Delta>1$. The unique example 
attaining $m(2,q)=q$, which geometrically corresponds to the points of a line, is $q$-divisible. For $k'\le 3$ all minimal 
binary codes of length $m(2k',2)$ are $2$-divisible and for dimension $k=8$ there are minimal binary codes of length $m(8,2)=24$ that are $2$-divisible while not all examples are of this type. 
In \cite{trifferent_kurz} it was shown that the unique minimal code attaining $m(5,3)=19$ is $3$-divisible. So, at least for the small parameters we have considered here there exist $q$-divisible
examples of minimum possible size $m(k,q)$ whenever the lower bound $(q-1)(k-1)+1$ on the minimum distance, see Theorem~\ref{thm_bound_summary}.(b), is divisible by $q$.\footnote{The second case 
where this condition is met, after the first $k=2$, is at dimension $k=q+2$.} We remark that also some constructions for minimal codes are based on few-weight codes, which often have a non-trivial 
divisibility constant, see e.g.\ \cite{mesnager2019several,shi2020several,shi2021two}. Due to the mentioned possible relations between minimal and divisible codes we introduce the minimum possible 
length $n=m(k,q;\Delta)$ of a $\Delta$-divisible minimal $[n,k]_q$-code. Here we initiate the study of $m(k,q;\Delta)$ and give bounds and exact values, both computationally and theoretically.

The remaining part of this paper is structured as follows. In Section~\ref{sec_preliminaries} we state the necessary preliminaries before we study bounds and exact values for $m(k,q;\Delta)$ 
in Section~\ref{sec_div}. For the special case of binary minimal codes with trivial divisibility $\Delta=1$ we study the minimum possible length $m(k,2;1)=m(k,2)$ in Section~\ref{sec_binary}.  

\section{Preliminaries}
\label{sec_preliminaries}

First we consider the well-known correspondence between (non-degenerated) $[n,k]_q$-codes and 
multisets of points in the projective space $\operatorname{PG}(k-1,q)$ of cardinality $n$, i.e., the columns of a 
generator matrix each generate a point, see e.g.\ \cite{dodunekov1998codes}. We represent each multiset of points in $\operatorname{PG}(v-1,q)$ 
by a mapping $M\colon\mathcal{P}\to\mathbb{N}_{\ge 0}$ from the set of points $\mathcal{P}$ in $\operatorname{PG}(v-1,q)$ to the non-negative integers, 
i.e., to each point $P$ we assign a multiplicity $M(P)$. We extend this notion to arbitrary subspaces $S$ by defining $M(S)$ as the sum over all 
point multiplicities $M(P)$ for all points $P$ in $S$. The cardinality of $M$, i.e., the sum of the multiplicities of all points, is denoted by $\# M$. 
We say that a multiset $M$ of points is \emph{spanning} if the points with positive multiplicity span the entire ambient space. 
\begin{definition}
  A multiset $M$ of points in a projective space is called a \emph{strong blocking multiset} if for every
  hyperplane $H$, we have $\langle S \cap H\rangle = H$.
\end{definition}
If $M$ is the multiset of points associated to a linear code $C$, then $C$ is minimal iff $M$ is a strong blocking multiset, see e.g.\ \cite{alfarano2022geometric,tang2021full}. 
Directly from the definition of a strong blocking multiset we can read off that a multiset of points in $\operatorname{PG(1,q)}$ is a strong blocking multiset 
iff it contains every point of the entire projective space. Clearly adding points to a multiset does not destroy the property of being a strong blocking multiset, 
so that we consider \emph{minimal strong blocking sets} in the following, i.e., set of points that are a strong blocking multiset but such that every proper 
subset is not a strong blocking multiset. So, in $\PG(1,q)$ the unique minimal strong blocking set is a line, so that 
\begin{equation}
  m(2,q)=q.
\end{equation}   
Since each linear code associated to the point set of a $k$-dimensional subspace over $\mathbb{F}_q$ is $q$-divisible, see e.g.\ \cite[Lemma 2.a]{kiermaier2020lengths}, 
we have
\begin{equation}
  m(2,q;q)=q  
\end{equation}
for each positive integer $\Delta$. For dimension $k=1$ we clearly have $m(1,q)=1$ and $m(1,q;\Delta)=\Delta$ for all $\Delta\in\mathbb{N}_{\ge 1}$.  

The representation of a linear code $C$ by a multiset of points $M$ is pretty useful. If we multiply the multiplicity $M(P)$ of every 
point $P$ by some positive integer $t$, the cardinality as well as the divisibility is increased by a factor of $t$. So, we have
\begin{equation}
  m(k,q)\le m(k,q;\Delta)\le \Delta \cdot m(k,q)
\end{equation}  
for all $\Delta\in \mathbb{N}_{\ge 1}$. Our examples for dimensions $1$ and $2$ show that both bounds can be attained with equality. Similarly, we have
\begin{equation}
  m(k,q;\Delta)\le m(k,q;t\cdot \Delta)\le  t\cdot m(k,q;\Delta)
\end{equation}
for all $\Delta,t\in \mathbb{N}_{\ge 1}$. If $t$ is coprime to $q$, then a $t$-divisible linear code over $\mathbb{F}_q$ is a $t$-fold repetition of a smaller code, 
see e.g.\ \cite[Theorem 1]{ward1981divisible}. So, we have
\begin{equation}
  \label{eq_ward}
  m(k,q;t\cdot \Delta)=  t\cdot m(k,q;\Delta)
\end{equation}
for all $t\in\mathbb{N}_{\ge 1}$ with $\gcd(q,t)=1$. For binary codes we can consider extension by a parity bit to conclude
\begin{equation}
  m(k,2;2)\le m(k,2;1)+1.
\end{equation} 

\medskip

Given a linear code $C$ the \emph{weight} $\operatorname{wt}(c)$ of a codeword $c\in C$ is the number of non-zero entries. With this, the minimum Hamming distance $d$ 
of $C$ is the minimum weight over all non-zero codewords of $C$. If an $[n,k]_q$-code has minimum Hamming distance $d$ then we also speak of an $[n,k,d]_q$-code. 
The polynomial $\sum_{c\in C} x^{\operatorname{wt}(c)}$ is called the \emph{weight enumerator} of $C$.
We summarize the current knowledge on general bounds for the length $n$, the minimum (non-zero) weight $w_{\min}$, and the maximum (non-zero) weight $w_{\max}$ of 
a minimal linear code as follows:
\begin{theorem}
  \label{thm_bound_summary}
  For each minimal $[n,k]_q$-code we have
  \begin{enumerate}
    \item[(a)] $n\ge(q+1)(k-1)$;
    \item[(b)] $d=w_{\min}\ge  (k-1)(q-1)+1$; and 
    \item[(c)] $w_{\max}\le n-k+1$.
  \end{enumerate}
\end{theorem} 
\begin{proof}
  For (a) see e.g.\ \cite[Theorem 2.14]{three_combinatorial_perspectives}, for (b) see e.g.\ \cite[Theorem 23]{heger2021short} or 
\cite[Theorem 2.8]{three_combinatorial_perspectives}, and for (c) see \cite[Proposition 1.5]{three_combinatorial_perspectives}. 
\end{proof}

\medskip

A linear code $C$ is called \emph{quasi-cyclic of index $l$} if the shift of $l$ positions to the right of every codeword is also a codeword, see e.g.\ \cite{ling2001algebraic}. 
The case $l=1$ corresponds to \emph{cyclic} codes.  A \emph{circulant matrix} is a square matrix of the form
$$
  G=\begin{pmatrix}
    g_0 & g_1 & \dots & g_{k-1} \\ 
    g_{k-1} & g_0 & \dots & g_{k-2} \\ 
    \vdots & \vdots & \ddots & \vdots \\ 
    g_1 & g_2 & \dots & g_0
  \end{pmatrix}
$$ 
and its \emph{associated polynomial} is given by $g(x)=g_0+g_1x+\dots+g_{k-1}x^{k-1}$, see e.g.\ \cite{kra2012circulant}. A circulant matrix $G$ generates a $[k,k']$ code, where $1\le k'\le k$. 
For circulant matrices $G_1,\dots,G_l$ the matrix $\begin{pmatrix}G_1&\dots&G_l\end{pmatrix}$ generates a quasi-cyclic code of index $l$, length $lk$, and dimension at most $k$. Reordering 
the matrices $G_i$ results in equivalent codes and every quasi-cyclic code admits such a representation. If e.g.\ $G_1$ has full rank, then there exist circulant matrices $G_2',\dots,G_l'$ 
such that $\begin{pmatrix}I&G_2'&\dots&G_l'\end{pmatrix}$ is a generator matrix of the same code. Heuristically, the assumption that at least one of the circulant matrices has full rank does 
not seem to exclude codes with good parameters, see e.g.\ \cite{heijnen1998some}. For $l=2$ one also speaks of a \emph{double circulant code}, see e.g.\ \cite[Chapter 16]{macwilliams1977theory}. 
Here we want to have a little bit more flexibility. 
\begin{definition}
Let $g\in\mathbb{F}_q^s$ and $u,v$ be positive integers that are divisible by $s$. A $u\times v$ circulant matrix with generator $g$ is a matrix $G\in\mathbb{F}_q^{u\times v}$ whose 
first row consists of $v/s$ copies of $g$ and every other row is obtained by a cyclic right shift of the row directly above it. 
\end{definition}
As an example, a $4\times 6$ circulant matrix over $\F_2$ with generator $\begin{pmatrix}1&0\end{pmatrix}$ is given by
$$
  \begin{pmatrix}
    10\,\,10\,\,10\\ 
    01\,\,01\,\,01\\[0.4mm]
    10\,\,10\,\,10\\ 
    01\,\,01\,\,01
  \end{pmatrix}\!,  
$$
where we have visualized the $2\times 2$ submatrices which a circulant with generator $g$. I.e.\ those submatrices are copied $v/s$ times to the right and $u/s$ times to the bottom.
\begin{definition}
  \label{def_generalized_circulant_matrix}
  A \emph{generalized circulant matrix of type $(\mathbf{\alpha},\mathbf{\beta},t)$} is a matrix of the form
  $$
    G=\begin{pmatrix}
      G_{11} & G_{12} & \dots  & G_{1b} \\ 
      G_{21} & G_{22} & \dots  & G_{2b} \\ 
      \vdots & \vdots & \ddots & \vdots \\ 
      G_{a1} & G_{a2} & \dots  & G_{ab} 
    \end{pmatrix}\!,
  $$
  where the $G_{ij}$ are $\left(u_{ij}\times v_{ij}\right)$ circulant matrices with generator $g_{ij}\in\mathbb{F}_2^{s_{ij}}$ such that
  \begin{itemize}
    \item all $u_{ij}\!$'s, $v_{ij}\!$'s, $s_{ij}\!$'s are divisors of $t$ and $s_{ij}=t$ occurs at least once;
    \item the number of rows $u_{ij}$ of $G_{ij}$ is the same for all $j$;
    \item the number of columns $v_{ij}$ of $G_{ij}$ is the same for all $i$;
    \item $\alpha=1^{\alpha_1} \dots t^{\alpha_t}$ and $u_{i1}=l$ occurs $\alpha_l$ times for all $1\le l\le t$; and
    \item $\beta=1^{\beta_1} \dots t^{\beta_t}$ and $v_{1j}=l$ occurs $\beta_l$ times for all $1\le l\le t$.
  \end{itemize}
  Moreover, we call $G$ \emph{systematic} if it starts with a full unit matrix.
\end{definition}

As an example we consider 
$$
  G=\begin{pmatrix}
  \underline{1}\,\,\underline{0}00000\,\,\underline{0}00000\,\,\underline{1}\,\,\underline{0}00000\,\,\underline{1}11111\,\,\underline{1}11111\\[0.5mm]
  \underline{0}\,\,\underline{100000}\,\,\underline{000000}\,\,\underline{1}\,\,\underline{001011}\,\,\underline{000101}\,\,\underline{001011}\\
  0\,\,010000\,\,000000\,\,1\,\,100101\,\,100010\,\,100101\\
  0\,\,001000\,\,000000\,\,1\,\,110010\,\,010001\,\,110010\\
  0\,\,000100\,\,000000\,\,1\,\,011001\,\,101000\,\,011001\\
  0\,\,000010\,\,000000\,\,1\,\,101100\,\,010100\,\,101100\\
  0\,\,000001\,\,000000\,\,1\,\,010110\,\,001010\,\,010110\\[0.5mm]
  \underline{0}\,\,\underline{000000}\,\,\underline{100000}\,\,\underline{0}\,\,\underline{000111}\,\,\underline{001011}\,\,\underline{111011}\\
  0\,\,000000\,\,010000\,\,0\,\,100011\,\,100101\,\,111101\\
  0\,\,000000\,\,001000\,\,0\,\,110001\,\,110010\,\,111110\\
  0\,\,000000\,\,000100\,\,0\,\,111000\,\,011001\,\,011111\\
  0\,\,000000\,\,000010\,\,0\,\,011100\,\,101100\,\,101111\\
  0\,\,000000\,\,000001\,\,0\,\,001110\,\,010110\,\,110111
  \end{pmatrix}\!,
$$
where the generators are underlined, $\alpha=1^1 6^2$, $\beta=1^2 6^5$, $t=5$, and $G$ is systematic. $G$ generates a $[32,13,10]_2$-code with 
weight enumerator $1+346x^{10}+860x^{12}+1636x^{14}+2405x^{16}+1840x^{18}+796x^{20}+268x^{22}+34x^{24}+6x^{26}$ that is indeed optimal. We remark 
that a $[32,13,10]_2$-code was first found by Shearer using a computer search, see \cite{brouwer1993updated}. Using the \texttt{Magma} command \texttt{BestKnownLinearCode} 
a corresponding generator matrix can be retrieved that generates a $[32,13,10]_2$-code with weight enumerator $1+348x^{10}+853x^{12}+1641x^{14}+2418x^{16}
+1805x^{18}+839x^{20}+235x^{22}+49x^{24}+3x^{26}$ and a trivial automorphism group.    

We neither claim that the notion of generalized circulant matrices is new nor that the chosen name is optimal. There is a vast literature on quasi-cyclic codes and different 
shapes with several circulant matrices have been studied, see e.g.\ \cite{esmaeili2009generalized}. For generalized quasi-cyclic codes we refer to e.g.\ \cite{guneri2017structure,muchtadi2022generalized} 
and the references therein. Our notion of a generalized circulant matrix in Definition~\ref{def_generalized_circulant_matrix} allows us to describe many of our newly discovered codes. 
On the other hand more restricted classes allow computational non-existence results. E.g.\ there does not exist a double circulant even $[40, 20, 10]_2$-code \cite{gulliver1997classification}.   

There is a another point of view how generalized circulant matrices can be described. For given field size $q$, dimension $k$, and length $n$ let $1\le t\le k$ be an integer and 
$\alpha_1,\dots,\alpha_t$ be non-negative integers such that $\sum_{i\,:\,i\vert t} \alpha_i=k$ and $\alpha_t\ge 1$. With this let $\pi$ be a permutation of $\{1,\dots,k\}$ with $\alpha_i$ cycles of length $i$. Similarly, let 
$\beta_i$ be non-negative integers such that $\sum_{i\,:\,i\vert t} \beta_i=n$, $\beta_t\ge 1$,  and $\varphi$ be a permutation of $\{1,\dots,n\}$ with $\beta_i$ cycles of length $i$. As an example we consider
$q=2$, $k=13$, $n=32$, $t=6$, $\alpha_1=1$, and $\alpha_6=2$, so that we can choose $\pi=(1)(2,3,4,5,6,7)(8,9,10,11,12,13)$. The action of $\pi$ on the elements of $\mathbb{F}_q^k$ can also be 
described by the multiplication with a matrix $M_\pi\in\mathbb{F}_q^{k\times k}$. In our example we have 
$$
  M_\pi=\begin{pmatrix}
    1\,\,000000\,\,000000\\[1mm] 
    0\,\,010000\,\,000000\\
    0\,\,001000\,\,000000\\
    0\,\,000100\,\,000000\\
    0\,\,000010\,\,000000\\
    0\,\,000001\,\,000000\\
    0\,\,100000\,\,000000\\[1mm]
    0\,\,000000\,\,010000\\
    0\,\,000000\,\,001000\\
    0\,\,000000\,\,000100\\
    0\,\,000000\,\,000010\\
    0\,\,000000\,\,000001\\
    0\,\,000000\,\,100000
  \end{pmatrix}
$$    
Using the geometric interpretation of a linear code $C$ as a multiset of point $M$ in $\PG(k-1,q)$, the action of $M_\pi$ partitions the set of points, as well as the 
set of hyperplanes, of $\PG(k-1,q)$ into orbits, whose lengths are divisors of $t$. Assuming the prescribed automorphism $M_\pi$ it is sufficient to state a representant of each 
chosen point orbit, which corresponds to a column for each block of the generator matrix $G$. The underlined generators in $G$ are just another parameterization. Note 
that $\langle M_\pi\rangle$ is a cyclic group and it is a common approach to search linear codes with good parameters by prescribing some group as a subgroup of the automorphism 
group as solutions of an integer linear problem, since both the number of variables and constraints is reduced, see e.g.\ \cite{braun2005optimal}. If we loop over suitable 
candidates for the generators in a generalized circulant matrix we can use the group action to partition the possible generators into orbits and also to restrict the minimum 
distance computations to codeword orbits. The condition $\beta_t\ge 1$ ensures that the corresponding codes have a cyclic automorphism of order $t$. In our example we have 
$\varphi=(1)(234567)(8\dots13)(14)(15\dots20)(21\dots26)(27\dots32)$. For more details on codes with a given automorphism and the relation to generalized quasi-cyclic codes we 
refer to \cite{bouyuklieva2024structure}.  

In Section~\ref{sec_binary} we will use the structure of generalized circulant matrices to find improved upper bounds for $m(k,2)$. As a spin-off of the underlying 
computer searches we also find a few codes improving upon the best known linear codes, see e.g.\ \cite{Grassl:codetables}. The following three matrices are generator matrices of $[50,20,13]_2$-, 
$[52,21,13]_2$-, and $[56,24,13]$-codes. respectively. 
$$
  \left(\begin{smallmatrix}
  10000000000000000000\,\,0101111111\,\,0001001101\,\,0000000111\\
  01000000000000000000\,\,1011111110\,\,0010011010\,\,0000001110\\
  00100000000000000000\,\,0111111101\,\,0100110100\,\,0000011100\\
  00010000000000000000\,\,1111111010\,\,1001101000\,\,0000111000\\
  00001000000000000000\,\,1111110101\,\,0011010001\,\,0001110000\\
  00000100000000000000\,\,1111101011\,\,0110100010\,\,0011100000\\
  00000010000000000000\,\,1111010111\,\,1101000100\,\,0111000000\\
  00000001000000000000\,\,1110101111\,\,1010001001\,\,1110000000\\
  00000000100000000000\,\,1101011111\,\,0100010011\,\,1100000001\\
  00000000010000000000\,\,1010111111\,\,1000100110\,\,1000000011\\[0.5mm]
  00000000001000000000\,\,1101100010\,\,0001010011\,\,0000100101\\
  00000000000100000000\,\,1011000101\,\,0010100110\,\,0001001010\\
  00000000000010000000\,\,0110001011\,\,0101001100\,\,0010010100\\
  00000000000001000000\,\,1100010110\,\,1010011000\,\,0100101000\\
  00000000000000100000\,\,1000101101\,\,0100110001\,\,1001010000\\
  00000000000000010000\,\,0001011011\,\,1001100010\,\,0010100001\\
  00000000000000001000\,\,0010110110\,\,0011000101\,\,0101000010\\
  00000000000000000100\,\,0101101100\,\,0110001010\,\,1010000100\\
  00000000000000000010\,\,1011011000\,\,1100010100\,\,0100001001\\
  00000000000000000001\,\,0110110001\,\,1000101001\,\,1000010010\\
  \end{smallmatrix}\right)  
$$
$$
  \left(\begin{smallmatrix}
  100000000000000000000\,\,0000010111\,\,0011001101\,\,0101011011\,\,1\\
  010000000000000000000\,\,0000101110\,\,0110011010\,\,1010110110\,\,1\\
  001000000000000000000\,\,0001011100\,\,1100110100\,\,0101101101\,\,1\\
  000100000000000000000\,\,0010111000\,\,1001101001\,\,1011011010\,\,1\\
  000010000000000000000\,\,0101110000\,\,0011010011\,\,0110110101\,\,1\\
  000001000000000000000\,\,1011100000\,\,0110100110\,\,1101101010\,\,1\\
  000000100000000000000\,\,0111000001\,\,1101001100\,\,1011010101\,\,1\\
  000000010000000000000\,\,1110000010\,\,1010011001\,\,0110101011\,\,1\\
  000000001000000000000\,\,1100000101\,\,0100110011\,\,1101010110\,\,1\\
  000000000100000000000\,\,1000001011\,\,1001100110\,\,1010101101\,\,1\\[0.5mm]
  000000000010000000000\,\,0011111011\,\,0011110101\,\,0111010100\,\,1\\
  000000000001000000000\,\,0111110110\,\,0111101010\,\,1110101000\,\,1\\
  000000000000100000000\,\,1111101100\,\,1111010100\,\,1101010001\,\,1\\
  000000000000010000000\,\,1111011001\,\,1110101001\,\,1010100011\,\,1\\
  000000000000001000000\,\,1110110011\,\,1101010011\,\,0101000111\,\,1\\
  000000000000000100000\,\,1101100111\,\,1010100111\,\,1010001110\,\,1\\
  000000000000000010000\,\,1011001111\,\,0101001111\,\,0100011101\,\,1\\
  000000000000000001000\,\,0110011111\,\,1010011110\,\,1000111010\,\,1\\
  000000000000000000100\,\,1100111110\,\,0100111101\,\,0001110101\,\,1\\
  000000000000000000010\,\,1001111101\,\,1001111010\,\,0011101010\,\,1\\[0.5mm]
  000000000000000000001\,\,1111111111\,\,1111111111\,\,1111111111\,\,1
  \end{smallmatrix}\right)  
$$
$$
  \left(\begin{smallmatrix}
  100000000000000000000000\,\,00000111\,\,00001011\,\,00011111\,\,00101111\\
  010000000000000000000000\,\,00001110\,\,00010110\,\,00111110\,\,01011110\\
  001000000000000000000000\,\,00011100\,\,00101100\,\,01111100\,\,10111100\\
  000100000000000000000000\,\,00111000\,\,01011000\,\,11111000\,\,01111001\\
  000010000000000000000000\,\,01110000\,\,10110000\,\,11110001\,\,11110010\\
  000001000000000000000000\,\,11100000\,\,01100001\,\,11100011\,\,11100101\\
  000000100000000000000000\,\,11000001\,\,11000010\,\,11000111\,\,11001011\\
  000000010000000000000000\,\,10000011\,\,10000101\,\,10001111\,\,10010111\\[0.5mm]
  000000001000000000000000\,\,10100001\,\,10011000\,\,01101110\,\,10101001\\
  000000000100000000000000\,\,01000011\,\,00110001\,\,11011100\,\,01010011\\
  000000000010000000000000\,\,10000110\,\,01100010\,\,10111001\,\,10100110\\
  000000000001000000000000\,\,00001101\,\,11000100\,\,01110011\,\,01001101\\
  000000000000100000000000\,\,00011010\,\,10001001\,\,11100110\,\,10011010\\
  000000000000010000000000\,\,00110100\,\,00010011\,\,11001101\,\,00110101\\
  000000000000001000000000\,\,01101000\,\,00100110\,\,10011011\,\,01101010\\
  000000000000000100000000\,\,11010000\,\,01001100\,\,00110111\,\,11010100\\[0.5mm]
  000000000000000010000000\,\,01000110\,\,10001111\,\,11101001\,\,10001101\\
  000000000000000001000000\,\,10001100\,\,00011111\,\,11010011\,\,00011011\\
  000000000000000000100000\,\,00011001\,\,00111110\,\,10100111\,\,00110110\\
  000000000000000000010000\,\,00110010\,\,01111100\,\,01001111\,\,01101100\\
  000000000000000000001000\,\,01100100\,\,11111000\,\,10011110\,\,11011000\\
  000000000000000000000100\,\,11001000\,\,11110001\,\,00111101\,\,10110001\\
  000000000000000000000010\,\,10010001\,\,11100011\,\,01111010\,\,01100011\\
  000000000000000000000001\,\,00100011\,\,11000111\,\,11110100\,\,11000110  
  \end{smallmatrix}\right)  
$$
These three codes imply the further improvements $[49,19,13]_2$, $[50,19,14]_2$, 
$[51,20,14]_2$, $[53,21,14]_2$, and $[57,24,14]_2$. For other recently improved linear codes based on circulant 
matrices we refer e.g.\ to \cite{yu2025construction}.

\begin{remark}
  As point out by Markus Grassl \cite{Grassl:codetables}, there are more compact description of the three stated improved binary linear codes. The $[50,20,13]_2$ and the $[56,24,13]_2$ 
  codes are quasi-cyclic codes while the $[52,21,13]_2$ code can be obtained by applying Construction XX \cite{alltop2003method} to quasi-cyclic codes. More precisely, the 
  $[50,20,13]_2$ code is a quasi-cyclic code of length $50$ stacked to height $2$ with generating polynomials $1$,  $0$, $x^9 + x^8 + x^7 + x^6 + x^5 + x^3 + x^2 + x$,  
  $x^9 + x^3 + x^2 + x$, $x^9 + x^6 + x^3$,  $0$,  $1$,  $x^8 + x^7 + x^6 + x + 1$,  $x^6 + x^5 + x^3 + x$,  $x^9 + x^8 + x^3$ and the $[56,24,13]_2$ code is a 
  quasi-cyclic code of length $56$ stacked to height $3$ with generating polynomials $1$,  $0$,  $0$,  $x^3 + x^2 + x$,  $x^4 + x^2 + x$,  $x^5 + x^4 + x^3 + x^2 + x$,  $x^6 + x^4 + x^3 + x^2 + x$,  
  $0$,  $1$,  $0$,  $x^6 + x + 1,  x^5 + x^4 + 1$,  $x^7 + x^6 + x^4 + x^3 + x^2$,  $x^6 + x^4 + x + 1$,  $0$,  $0$,  $1$,  $x^7 + x^3 + x^2$,  $x^4 + x^3 + x^2 + x + 1$,  
  $x^7 + x^6 + x^4 + x + 1,  x^4 + x^3 + x + 1$.
\end{remark}

\section{Minimum lengths of divisible minimal codes}
\label{sec_div}

In this section we consider the determination of the smallest possible length $n=m(k,q;\Delta)$ of a minimal $\Delta$-divisible $[n,k]_q$-code. For dimensions $k\le 2$ the 
results are stated easily using the geometric reformulation of linear codes as multisets of points. Clearly, we have $m(1,q;\Delta)=\Delta$ attained by a $\Delta$-fold point. 
For dimension $k=2$ each point has multiplicity at least $1$ since the code has to be minimal. From $\Delta$-divisibility we conclude that the point multiplicities are pairwise congruent
modulo $\Delta$, so that the minimum possible length is attained if all point multiplicities are equal. Thus, we have $m(2,q;\Delta)=\tfrac{(q+1)\Delta}{q}$ if $\Delta$ is divisible by $q$ 
(attained by a $\Delta/q$-fold line) and $m(2,q;\Delta)=(q+1)\Delta$ (attained by a $\Delta$-fold line). Due to Equation~(\ref{eq_ward}) it suffices to consider the cases where $\Delta$ does 
not contain a non-trivial factor $t$ that is coprime to the field size $q$.
 
If the divisibility constant is large enough, when considering power of the characteristic 
only, we can give a precise answer:
\begin{proposition}
  For $r\ge k-1$ we have $m\!\left(k,q;q^r\right)=q^{r-k+1}\cdot \frac{q^k-1}{q-1}$.
\end{proposition}
\begin{proof}
  Since the code is $q^r$-divisible we have $d\ge q^r$, so that we can apply the Griesmer bound for the lower bound. An attaining example is given by the $q^{r-k+1}$-fold full $k$-space.
\end{proof}  
\begin{proposition}
  \label{prop_div_dim_one_less}
  For $k\ge 2$ we have $m\!\left(k,2;2^{k-2}\right)=2^k-1$.
\end{proposition}
\begin{proof}
  Since the $k$-dimensional simplex code is $2^{k-1}$-divisible and minimal, we have $m\!\left(k,2;2^{k-2}\right)\le 2^k-1$, so that we assume $n\le 2^k-1$ for the length of an attaining code 
  $C$. Note that the possible non-zero weights of $C$ are given by $i\cdot 2^{k-2}$ for $1\le i\le 3$. 
  
  If $c\in C$ is a codeword of weight $3\cdot 2^{k-2}$, then the corresponding residual code $C_c$ has length at most $2^{k-2}-1$ and dimension $k-1$ (since $C$ is minimal). Thus, we have $k\ge 3$ 
  and $C_c$ is $2^{k-3}$-divisible with $2^{k-3}$ as the unique non-zero weight. Since one-weight codes are repetitions of simplex codes, see e.g.~\cite{bonisoli1984every}, $C_c$ can 
  have  dimension of at most $k-2$ --- contradiction.
  
  So, let $a_1$ be the number of codewords of weight $2^{k-2}$ and $a_2$ be the number of codewords of weight $2^{k-1}$. From the first two MacWilliams equations we compute $a_1+a_2=2^k-1$ and 
  $2n=a_1+2a_2$, so that $a_1=2^{k+1}-2-2n$, i.e., $a_1$ is even. Since the code is minimal, the sum of any two different codewords of weight $2^{k-2}$ has again weight $2^{k-2}$, i.e.\ the codewords 
  of the smallest weight form subcode and we have $a_1=2^t-1$ for some integer $t$.\footnote{We remark that $\Delta$-divisible linear codes spanned by codewords of weight $\Delta$ have been completely 
  classified in \cite{kiermaier2023classification}. Note that there exists a $2^{k-2}$-divisible linear code of length $2^{k-1}$ and dimension $k$ satisfying $a_1=2^k-2$, $a_2=1$. However, this code, 
  corresponding to an affine subspace, is not minimal.} Thus, we have $t=0$ and $a_1=0$, i.e., we have $d\ge 2^{k-1}$ for the minimum distance and can apply the Griesmer bound for the lower bound 
  $n\ge 2^k-1$.   
\end{proof}

For parameters not covered by these two propositions and dimension $k\ge 3$ we have applied the software \texttt{LinCode} for the enumeration of linear codes 
\cite{bouyukliev2021computer} using the bounds for the minimum and maximum possible weight in Theorem~\ref{thm_bound_summary} and also using the weight restrictions 
implied by the divisibility constant $\Delta$. For field sizes $q=2$ and $q=3$ we summarize our numerical results in Table~\ref{table_div_q_2_3}. With this, $m(k,q;\Delta)$ is
completely determined for $k\le 9$ if $q=2$ and for $k\le 5$ if $q=3$. 

\begin{table}[htp]
  \begin{center}
    \begin{tabular}{lrrrrrrrrrrrrrr}
    \hline
    $k$                & 4 & 4 &  5 &  5 &  5 &  6 &  6 &  6 &  6 &  7 &  7 &  7 &  7 &  7 \\
    $q$                & 2 & 2 &  2 &  2 &  2 &  2 &  2 &  2 &  2 &  2 &  2 &  2 &  2 &  2 \\
    $\Delta$           & 1 & 2 &  1 &  2 &  4 &  1 &  2 &  4 &  8 &  1 &  2 &  4 &  8 & 16 \\
    $m(k,q;\Delta)$    & 9 & 9 & 13 & 14 & 17 & 15 & 15 & 18 & 36 & 20 & 21 & 26 & 42 & 84 \\
    \hline\\[2mm]
    \hline
    $k$                &  8 &  8 &  8 &  8 &  8 &   8 &  9 &  9 &  9 &  9 &  9 &   9 &   9 & 10 \\
    $q$                &  2 &  2 &  2 &  2 &  2 &   2 &  2 &  2 &  2 &  2 &  2 &   2 &   2 &  2 \\
    $\Delta$           &  1 &  2 &  4 &  8 & 16 &  32 &  1 &  2 &  4 &  8 & 16 &  32 &  64 &  4 \\
    $m(k,q;\Delta)$    & 24 & 24 & 29 & 45 & 90 & 174 & 26 & 27 & 30 & 58 & 96 & 192 & 384 & 31 \\
    \hline\\[2mm]
    \hline
    $k$                & 10 & 10 &  10 &  10 & 3 &  3 &  4 &  4 &  4 &  5 &  5 &  5 &   5 \\
    $q$                &  2 &  2 &   2 &   2 & 3 &  3 &  3 &  3 &  3 &  3 &  3 &  3 &   3 \\
    $\Delta$           &  8 & 16 &  32 &  64 & 1 &  3 &  1 &  3 &  9 &  1 &  3 &  9 &  27 \\
    $m(k,q;\Delta)$    & 60 & 93 & 186 & 366 & 9 & 12 & 14 & 15 & 38 & 19 & 19 & 48 & 116 \\
    \hline    
    \end{tabular}
    \caption{Exact values of $m(k,q;\Delta)$ for small parameters where $q\in\{2,3\}$.}
    \label{table_div_q_2_3}
  \end{center}
\end{table}

\begin{lemma}
  \label{lemma_su2}
  For each integer $t\ge 2$ we have $m\!\left(2t,2;2^{t-1}\right)\le 3\cdot \left(2^t-1\right)$.
\end{lemma}
\begin{proof}
  Consider the linear code $C$ corresponding to three pairwise disjoint $t$-dimensional subspaces of $\PG(2t-1,2)$. With this, $C$ is an 
  $\left[3\cdot \left(2^t-1\right),2t\right]_2$-code with non-zero weighs $2\cdot 2^{t-1}$ and $3\cdot 2^{t-1}$, which is minimal due to the 
  Ashikhmin-Barg condition \cite{ashikhmin1998minimal}. 
\end{proof}
We remark that the constructed projective two-weight code contains to the family SU2 in \cite{calderbank1986geometry}. While equality is attained in Lemma~\ref{lemma_su2} 
for $t\in \{2,4,5\}$, we have $m(6,2;4)=18<21$.

\medskip

The interesting codes, i.e.\ those that cannot be obtained by repetitions of smaller codes, are given by
$$
  \begin{pmatrix}
    11111111111010000\\
    00000111111101000\\
    00111000111100100\\
    01011011001100010\\
    11100001011100001
  \end{pmatrix} 
$$
attaining $m(5,2;4)=17$ with weight enumerator $1+25x^{8}+6x^{12}$ and an automorphism group of order $720$, as well as      
$$
  \begin{pmatrix}
    111111111110100000\\
    000001111111010000\\
    001110001111001000\\
    010110110011000100\\
    111000010111000010\\
    011011100101000001
  \end{pmatrix}
$$
attaining $m(6,2;4)=18$ with weight enumerator $1+45x^{8}+18x^{12}$ and an automorphism group of order $2160$, see \cite{bierbrauer1997family}. 
For the first code we remark that the automorphism group is isomorphic to the symmetric group $S_6$ and has point orbits in $\PG(4,2)$ of sizes $1$, $15$ and $15$. The 
unique point has multiplicity $2$ in the attaining construction and the points in one of the other classes have multiplicity $1$. 
The unique code attaining $m(7,2;8)=42$ is given by
$$
\begin{pmatrix}
  111111111111111111111110000000000001000000\\
  000000000001111111111111111111111100100000\\
  000000000110000001111110000011111110010000\\
  000000001010001110001110011100011110001000\\
  111111111100110010010110101101100110000100\\
  000000010011000110110010110110101010000010\\
  000000100010011010111001011010100110000001
\end{pmatrix} 
$$
with weight enumerator $1+45x^{16}+82x^{24}$ and an automorphism group of order $138240$. Considered as a multiset of points in $\PG(6,2)$ the automorphism group forms three point orbits 
of sizes $1$, $36$, and $90$ with point multiplicities $6$, $1$, and $0$, respectively. 
There are $62$ non-isomorphic doubly-even minimal $[29,8]_2$-codes. One example
is given by
$$
  \begin{pmatrix}
    11111111111111100000010000000\\
    00000001111111111100001000000\\
    00011110000111100011100100000\\
    00100110111001100101100010000\\
    01011011001010101110000001000\\
    11001001010011110010100000100\\
    01110010010011111001000000010\\
    00111000100101011101100000001
  \end{pmatrix}
$$ 
with weight enumerator $1+114x^{12}+119x^{16}+22x^{20}$ and an automorphism group of order $3$. 

There are two non-isomorphic $8$-divisible minimal $[45,8]_2$-codes. Both  have weight enumerator $1+45x^{16}+210x^{24}$ and are 
projective two-weight codes, see \cite{calderbank1986geometry} for more details. One example is given by the construction in Lemma~\ref{lemma_su2}. 
The orders of the automorphism groups are $3628800$ and $120960$.  
The unique code attaining $m(8,2;32)=174$ 
 is given by
 $$
   {\tiny
   \left(\begin{smallmatrix}
     1111111111111111111111111111111111111111111111111111111111111111111111111111111111111111111111\\ 10000000000000000000000000000000000000000000000000000000000000000000000010000000\\
     0000000000000000000000000000000000000000000000011111111111111111111111111111111111111111111111\\ 11111111111111111111111111111111111111111111111100000000000000000000000001000000\\
     0000000000000000000000000000000111111111111111100000000000000001111111111111111111111111111111\\ 10000000000000001111111111111111111111111111111111111111111111100000000000100000\\
     0000000000000000000000011111111000000001111111100000000111111110000000000000000000000001111111\\ 10000000111111110000000000000000000000001111111100000001111111111111110000010000\\
     0000000000000000000001100000011000000110000001100000011000000110000001111111111111111110000001\\ 10000011000000110000001111111111111111110000001100000110000001100000111000001000\\
     0000000000000000000110000111111001111110000110000001100001111110011110000000000000000110000110\\ 00001100001111110011110000000000000000110000110001111110000110001111110100000100\\
     0011111111111111111010001011111010001000111011101110111010001000101111111111111111111110001010\\ 01110111010001000101110000000000000000110001010010001000111011110111110100000010\\
     1100000000000000001110111101101110111011011110110111101110111011110111111111111111111010111110\\ 10010001000010010000011111111111111111010010000100010010001000100001011000000001
   \end{smallmatrix}\right)} 
 $$
 with 
weight enumerator $1+69x^{64}+186x^{96}$ and an automorphism group of order $61931520$. One of the five codes attaining $m(9,2;2)=27$ is given by 
$$
  \begin{pmatrix}
    111111111110000000100000000\\
    000001111111111100010000000\\
    001110001110011111001000000\\
    010110010010101101000100000\\
    111000110101100111000010000\\
    110011010001110001000001000\\
    001100111001001110000000100\\
    101010111000011001000000010\\
    011111010011011010000000001
  \end{pmatrix}
$$
with weight enumerator $1+90x^{10}+164x^{12}+84x^{14}+123x^{16}+50x^{18}$ and an automorphism group of order $48$. There are $9$ non-isomorphic codes attaining $m(9,2;4)=30$. All of them have 
weight enumerator $1+190x^{12}+255x^{16}+66x^{20}$. An example with an automorphism group of order $10$ is given by
$$
  \begin{pmatrix}
    111111111111111000000100000000\\
    000000011111111111000010000000\\
    000111100001111000111001000000\\
    001011100110011011001000100000\\
    011101101010001100011000010000\\
    101010110000001101111000001000\\
    010011010000110111101000000100\\
    111010000111001110010000000010\\
    111001001100011101100000000001
  \end{pmatrix}.
$$
There are $3$ non-isomorphic codes attaining $m(9,2;8)=58$. All of them have minimum distance $d=24$. An example with weight enumerator $1+194x^{24}+311x^{32}+6x^{40}$
and an automorphism group of order $384$ is given by
$$
  \begin{pmatrix}
    1111111111111111111111111111111000000000000000000100000000\\
    0000000000000000000111111111111111111111110000000010000000\\
    0000000000000111111000000111111000001111111111100001000000\\
    0000000111111000011111111000011000110000110001110000100000\\
    0001111000011001111001111111111011110000110111101000010000\\
    0110000011100011100000001011100101010111110010011000001000\\
    1110001000101110101000110011101010101111000000010000000100\\
    0011110001100111100010010000000010100111011110110000000010\\
    0001110001110000110000000100111111010111101100001000000001
  \end{pmatrix}.
$$
The unique code attaining $m(9,2;16)=96$ is given by
$$
  \left(\begin{smallmatrix}
    111111111111111111111111111111111111111111111110000000000000000000000000000000000000000100000000\\
    000000000000000000000001111111111111111111111111111111111111111111111100000000000000000010000000\\
    000000000001111111111110000000000001111111111110000000000011111111111111111111111000000001000000\\
    000001111110000001111110000001111110000001111110000011111100000011111100000111111111110000100000\\
    011110011110011110011110000110000110000110000110111100111100111100111100011000011000111000010000\\
    101110101110101110101110001010001010001010001011011101011101011101011100101000101001011000001000\\
    000110110111100110100010010010011110111010011000001101101111001101000111101001100011111000000100\\
    110011111010010000100110101100111110001000110100000110010110100111011101101111011110010000000010\\
    001010000110010010010100011110101111010111010111101101111010110110101110010010010010011000000001
  \end{smallmatrix}\right)
$$    
with weight enumerator $1+18x^{32}+472x^{48}+21x^{64}$ and an automorphism group of order $41472$.
There are two codes attaining $m(10,2;4)=31$. Both have weight enumerator $1+310x^{12}+527x^{16}+186x^{20}$, an automorphism group of order $155$, 
and are distance-optimal. Corresponding generator matrices are given by
$$
  \begin{pmatrix}
    1111111111111110000001000000000\\
    0000000111111111110000100000000\\
    0001111000011110001110010000000\\
    0010111001100110110010001000000\\
    0111011010100011000110000100000\\
    1010101100000011011110000010000\\
    1111010101101011111010000001000\\
    1011100000101111100100000000100\\
    0111101011001000011010000000010\\
    1110110000111000010110000000001
  \end{pmatrix}
  and
  \begin{pmatrix}
    1111111111111110000001000000000\\
    0000000111111111110000100000000\\
    0001111000011110001110010000000\\
    0010111001100110110010001000000\\
    0111011010100011000110000100000\\
    1010101100000011011110000010000\\
    0100110100001101111010000001000\\
    1110100001110011100100000000100\\
    1001010001001011110110000000010\\
    0011100111010100011010000000001
  \end{pmatrix}\!.
$$  
\cite[Chapter 8]{macwilliams1977theory} contains a construction of an infinite family of $\left(2^m-1, 2m\right)$ cyclic codes with
three different nonzero weights is given for odd $m$. As observed in \cite[Example 6]{cohen1985linear}, choosing $m=5$ yields a $4$-divisible 
minimal $[31,10,12]_2$ three-weight code. 
For $m(10,2;2)$ we have verified that length $28$ cannot be attained. 
There are three codes attaining $m(10,2;8)=60$, all with weight enumerator $1+270x^{24}+735x^{32}+18x^{40}$. The example with an automorphism group of order $69120$ 
is given by
$$
  \begin{pmatrix}
  111111111111111111111111111111100000000000000000001000000000\\
  000000000000000111111111111111111111110000000000000100000000\\
  000000011111111000000001111111100000001111111000000010000000\\
  000000000001111000011110000111100011110001111111000001000000\\
  000111111111111000011110011001101100110110011011100000100000\\
  011001100110011001100111100110000000000001111011010000010000\\
  101010100001100000101000001011100011001111101100110000001000\\
  101010101110111110011110100101110001000111111001110000000100\\
  101010111000000010011010111110101111110010111111110000000010\\
  101010100010001111111001101000100100010001100010110000000001
  \end{pmatrix}\!.
$$

The unique code attaining $m(10,2;16)=93$ is given by 
the construction in Lemma~\ref{lemma_su2}. 
It has 
weight enumerator $1+93x^{32}+930x^{48}$ and an automorphism group of order $59996160$.

The unique code attaining $m(10,2;64)=366$ is given by
$$
\tiny
\left(\begin{smallmatrix}
11111111111111111111111111111111111111111111111111111111111111111111111111111111111111111111111111111111111111111111111111111111111111111111\\ 
\quad 11111111111111111111111111111111111111111111111111100000000000000000000000000000000000000000000000000000000000000000000000000000000000\\
\quad 00000000000000000000000000000000000000000000000000000000000000000000000000000000001000000000\\
00000000000000000000000000000000000000000000000000000000000000000000000000000000000000000000000111111111111111111111111111111111111111111111\\
\quad 11111111111111111111111111111111111111111111111111111111111111111111111111111111111111111111111111111111111111111111111111111111111111\\
\quad 11111111111100000000000000000000000000000000000000000000000000000000000000000000000100000000\\
00000000000000000000000000000000000000000000000111111111111111111111111111111111111111111111111000000000000000000000000000000000000000000000\\
\quad 00011111111111111111111111111111111111111111111111100000000000000000000000000000000000000000000000111111111111111111111111111111111111\\
\quad 11111111111111111111111111111111111111111111111111111111111000000000000000000000000010000000\\
00000000000000000000000000000001111111111111111000000000000000000000000000000001111111111111111000000000000000000000000000000001111111111111\\ 
\quad 11100000000000000000000000000000000111111111111111100000000000000000000000000000001111111111111111000000000000000000000000000000001111\\
\quad 11111111111100000000000000000000000000000001111111111111111111111111111111000000000001000000\\
00011111111111111111111111111110000111111111111000000000000111111111111111111110000000000001111000000000000000000001111111111110000111111111\\
\quad 11100000000000000000000000000001111000000000000111100000000000111111111111111111110000000000001111000011111111111111111111111111110000\\
\quad 11111111111100000000000000000001111111111110000111111111111000000000001111111000000000100000\\
00000000000111111111111111111110000000000001111000000001111000011111111111111110000000011110000000000000000000000000000000011110000000000001\\
\quad 11100000000000000000000000011110000000000001111000000000001111000011111111111111110000000011110000000000000000111111111111111111110000\\
\quad 00000000111100000000000000000000000000011110000000000001111000000011110000000111000000010000\\
01100111111111111111111111111110011001111111111000000110000001111111111111111110000001100000011001111111111111111110011111111110011001111111\\
\quad 11100000011111111111111111100000011000000110000001100000110000001100000000000000000000001100000011001100111111000000000000000011110011\\
\quad 00111111111100000000000000000110011111111110011001111111111000001100000011011000100000001000\\
00000000011000000000000000000110000000000110011000011000011000011111111111111110000110000110000111111111111111111111100111100111111110011110\\
\quad 01100111100000000000000001100111111001111110011111101111110011111111111111111111110011111100111111111111001111000000000000000000111111\\
\quad 11001111001101111111111111111000000001100110000000000110011000110000110000000011010000000100\\
00000000101011111111111111111011111110111010101000101000101000011111111111111110111011101011111110000000000000000111101011101010000000100010\\
\quad 10101011100000000000000001101011111010001000101000010111110101111111111111111111110100010001010000111111010111111111111111111101010000\\
\quad 00010001010100000000000000000000000010101011111110111010101111011101011111000101010000000010\\
10101001000001111111111111111000101010001000000001000010000010100000000000000000001000100000101010000000000000000010100000100000101010000100\\
\quad 00000000111111111111111110100000101000010010000010100001010000010111111111111111110000100100000101010101000001000000000000000000000101\\
\quad 01000010000011111111111111111010100100000000101010001000000001000100000101101000100000000001
\end{smallmatrix}\right)
$$
with weight enumerator $1+141x^{128}+882x^{192}$ and an automorphism group of order $27745320960$.

\medskip
  
The unique code attaining $m(3,3;3)=12$ is given by
$$
\begin{pmatrix}
  111111110100\\
  000011221010\\
  011200022001
\end{pmatrix} 
$$
with weight enumerator $1+6x^{6}+20x^{9}$ and an automorphism group of order $48$. For $m(4,3;3)=15$ there are two attaining non-isomorphic codes. They 
are two-weight codes with weight enumerator $1+50x^{9}+30x^{12}$ and belong to the families FE1 and FE4 in \cite{calderbank1986geometry}.
The unique code attaining $m(4,3;9)=38$ is given by
$$
\begin{pmatrix}
  11111111111111111111111111000000001000\\
  00000000111111111222222222111111110100\\
  00000012000012222000011112000122220010\\
  01111200001200111001201110012000000001
\end{pmatrix} 
$$
with weight enumerator $1+12x^{18}+68x^{27}$ and an automorphism group of order $384$. 
The unique code attaining $m(5,3;9)=48$ is given by
$$
\begin{pmatrix}
  111111111111111111111111111111111110000000010000\\
  000000000000000001111111112222222221111111101000\\
  000000000111111120000111120111122220001222200100\\
  000000001011111210111001202012211121222000200010\\
  000000000200122202012222221221201210022012200001
\end{pmatrix} 
$$
with weight enumerator $1+6x^{18}+92x^{27}+144x^{36}$ and an automorphism group of order $96$. 
The unique code attaining $m(5,3;27)=116$ is given by
$$
\left(\begin{smallmatrix}
  11111111111111111111111111111111111111111111111111111111111111111111111111111111000000000000000000000000000000010000\\
  00000000000000000000000000111111111111111111111111111222222222222222222222222222111111111111111111111111110000001000\\
  00000000000000111111222222000000111111222222222222222000000111111111111111222222000001111112222222222222221111100100\\
  00000000012222012222012222011112011112000000000011112000012000000000000012000012122220122220000000000122220001200010\\
  00000000110112220122000012101121201222000000000000120011211000000000012222001200000121101120000000002201221121100001
\end{smallmatrix}\right) 
$$
with weight enumerator $1+30x^{54}+212x^{81}$ and an automorphism group of order $89856$.

\medskip

For $q=4$ also fractional powers of the field size need to be considered. For small parameters we have obtained $m(3,4;1)=12$, $m(3,4;2)=14$, $m(3,4;4)=15$, $m(3,4;8)=21$, $m(4,4;1)=18$, 
$m(4,4;2)=19$, $m(4,4;4)=20$, $m(4,4;8)=40$, $m(4,4;16)=62$, and $m(4,4;32)=85$. 
As the number suggest, we have a similar result as Proposition~\ref{prop_div_dim_one_less} for $q=4$: 
\begin{proposition}
  For $k\ge 2$ we have $m\!\left(k,4;2^{2k-3}\right)=\tfrac{4^k-1}{3}$.
\end{proposition}
\begin{proof}
  Since the $k$-dimensional simplex code is $4^{k-1}$-divisible and minimal, we have $m\!\left(k,4;2^{2k-3}\right)\le \tfrac{4^k-1}{3}$.  
  The possible non-zero weights of an attaining  code $C$ are given by $i\cdot 2^{2k-3}$ for $1\le i\le 2$. By $3a_i$ we denote the corresponding number of codewords,
  so that the first two MacWilliams equations yield $a_1+a_2=\tfrac{4^k-1}{3}$ and $2n=a_1+2a_2$. With this, $a_1=2\cdot\tfrac{4^k-1}{3}-2n$ is even. However, the assumption that $C$ is 
  minimal implies that the sum of any two different codewords with weight $\Delta:=2^{2k-3}$ also has weight $\Delta$. Thus, the codewords of weight $\Delta$ form a subcode implying that
  $a_1=\tfrac{4^t-1}{3}$ for some integer $t$.\footnote{We remark that $\Delta$-divisible linear codes spanned by codewords of weight $\Delta$ have been completely 
  classified in \cite{kiermaier2023classification}.} With this we conclude $t=0$ and $a_1=0$, i.e., we have $d\ge 4^{k-1}$ for the minimum distance and can apply the Griesmer bound for the 
  lower bound $n\ge \tfrac{4^k-1}{3}$.
\end{proof}

\section{Minimum lengths of binary minimal codes}
\label{sec_binary}

As introduced before, we denote by $m(k,q)$ the minimum possible length $n$ of a minimal $[n,k]_q$-code. In this section we will consider binary minimal codes only. The values 
$m(1,2)=1$, $m(2,2)=3$, $m(3,2)=6$, $m(4,2)=9$, $m(5,2)=13$, and $m(6,2)=15$ are known since a while, see \cite{sloane1993covering}; c.f.\ also \cite[Table 1]{dela2021maximum} 
and \cite{alfarano2022geometric}. The bounds $19\le m(7,2)\le 21$, $m(8,2)\le 25$, $m(9,2)\le 29$ were reported in \cite{sloane1993covering}.\footnote{The authors of \cite{bishnoi2023blocking} 
have determined $m(7,2)=20$ and $m(8,2)\le 24$ via ILP computations -- personal communication.} For $m(10,2)\le 30$ we refer to \cite[Section II.A]{cohen1994intersecting}. Constructions 
from \cite{bartoli2023small} yield $m(12,2)\le 42$, $m(15,2)\le 54$, $m(16,2)\le 63$, and \cite{sloane1993covering} states $m(11,2)\le 41$, $m(13,2)\le 51$, $m(17,2)\le 63$.  

As rigorously analyzed in \cite{scotti2023lower}, the lower bound $m(k,q)\ge (q+1)(k-1)$ (see Theorem~\ref{thm_bound_summary}.(a)) cannot be attained if $k$ is sufficiently 
large since the minimum distance $d\ge  (k-1)(q-1)+1=k$ (see Theorem~\ref{thm_bound_summary}.(b)) cannot be attained with equality for $n=(q+1)(k-1)$; c.f.\ \cite[Theorem 4]{sloane1993covering}. Indeed, the data at 
\url{www.codetables.de} on possible minimum distances of $[n,k]_2$-codes implies $m(9,2)\ge 26$, $m(10,2)\ge 28$, $m(11,2)\ge 31$, $m(12,2)\ge 34$, $m(13,2)\ge 39$, $m(14,2)\ge 41$, 
$m(15,2)\ge 45$, $m(16,2)\ge 47$, and $m(17,2)\ge 51$. We remark that \cite{scotti2023lower} also contains theoretical proofs for $m(k,2)>3(k-1)$ for $k\in\{5, 7, 8, 9, 11,13\}$. 

\begin{table}[htp]
  \begin{center}
    \begin{tabular}{lrrrrrrrrr}
    \hline
    $k$                & 1 & 2 & 3 & 4 &  5 &  6 &  7 &  8 &  9 \\
    $m(k,2)$           & 1 & 3 & 6 & 9 & 13 & 15 & 20 & 24 & 26 \\
    \hline\\[2mm]
    \hline
    $k$                & 10     & 11     & 12     & 13     & 14     & 15     & 16     & 17     \\
    $m(k,2)$           & 28--29 & 31--35 & 34--38 & 39--43 & 41--48 & 45--52 & 47--56 & 51--62 \\
    \hline
    \end{tabular}
    \caption{Bounds for $m(k,2)=m(k,2;1)$ for $k\le 17$.}
    \label{table_bounds_binary}
  \end{center}
\end{table}

Here we determine $m(7,2)=20$, $m(8,2)=24$, and $m(9,2)=26$, as well as full classifications of all codes attaining $m(k,2)$ for $k\le 7$ and those attaining $m(9,2)$. 
For $10\le m\le 17$ we give constructions improving the upper bounds for $m(k,2)$, see Table~\ref{table_bounds_binary}. 

\medskip

For $k\le 4$ the attaining examples are unique up to equivalence and have 
nice geometric descriptions, i.e., the corresponding strong blocking sets are given by a point, a line, a plane minus a point, and a hyperbolic quadric. Theoretical uniqueness proofs are 
pretty simple for $k\le 3$ and for $k=4$ we refer to \cite{smaldore2023all}. Alternatively we can describe the example for $k=4$ as the union of three disjoint lines.\footnote{A sketch of a direct 
uniqueness proof is given as follows. The standard equations for a projective $[n,4]_2$ code with minimum weight $4$ and maximum weight $n-3$ yield $n\ge 9$ and weight enumerator 
$1+9x^{4}+6x^{6}$ for $n=9$. Thus, the complement is a $2$-divisible projective code of length $6$ and dimension $k$, which has to be the union of two disjoint lines, see e.g.\
\cite[Proposition 17]{korner2023lengths}.}                     
The next value $m(5,2)=13$ is attained by exactly two non-equivalent codes given e.g.\ by generator matrices
$$
\begin{pmatrix}
1111110010000\\
0001111101000\\
1110010100100\\
0010101100010\\
0101010100001
\end{pmatrix}
\quad\text{and}\quad
\begin{pmatrix}
1111111010000\\
0001111101000\\
0110011100100\\
1010101100010\\
0101110000001
\end{pmatrix}\!\!.
$$
The corresponding weight enumerators and orders of the automorphism groups are given by 
$1+8x^{5}+8x^{6}+4x^{7}+7x^{8}+4x^{9}$, $1+6x^{5}+12x^{6}+4x^{7}+3x^{8}+6x^{9}$ and $8$, $48$, respectively. For $m(6,2)=15$ there is again a unique example 
given e.g.\ by the generator matrix
$$
  \begin{pmatrix}
    111111100100000\\
    000111110010000\\
    011001101001000\\
    100011101000100\\
    001110101000010\\
    011010110000001
  \end{pmatrix} \!\!
$$
of a BCH code, see \cite{cohen1985linear}. This code has weight enumerator $1+30x^{6}+15x^{8}+18x^{10}$ and an automorphism group of order $360$.  For a description of 
this code as the concatenation of two codes we refer to \cite{bartoli2023small}.

We remark that all above extremal codes meet the bounds for the minimum weight $w_{\min}\ge  (k-1)(q-1)+1=k$ (see Theorem~\ref{thm_bound_summary}.(b)) and the maximum 
weight $w_{\max}\le n-k+1$ (see Theorem~\ref{thm_bound_summary}.(c)). Using these bounds we have applied the software \texttt{LinCode} for the enumeration of linear codes 
\cite{bouyukliev2021computer} to determine $m(7,2)=20$ and $m(8,2)=24$. For $k=7$ there are $33$ non-equivalent extremal codes (all with $w_{\min}=7$ and $w_{\max}=14$). 
Generator matrices for those with more than eight automorphisms are given by                       
$$
\begin{pmatrix}
11111111100001000000\\
00001111111100100000\\
00110011101110010000\\
01010101110110001000\\
11011000110100000100\\
10001000111010000010\\
11110010010010000001  
\end{pmatrix}\!\!,
\begin{pmatrix}
11111111100001000000\\
00001111111100100000\\
00110011101110010000\\
01010100111110001000\\
10111001110100000100\\
11100101110010000010\\
11000110100110000001
\end{pmatrix}\!\!,
\begin{pmatrix}
11111111100001000000\\
00001111111100100000\\
00110011101110010000\\
01010100111110001000\\
10111001110110000100\\
11101010001100000010\\
11001001011010000001
\end{pmatrix}\!\!,
\begin{pmatrix}
11111111100001000000\\
00001111111100100000\\
00110011101110010000\\
01011101100110001000\\
11111100111010000100\\
10110100100110000010\\
01001101011010000001
\end{pmatrix}\!\!.
$$ 
We remark that there are $88010$ minimal $[22,7,8]_2$-codes. None of them can be extended to a minimal $[23,8,8]_2$-code. 
There are e.g.\ $2778120$ minimal $[22,6,8]_2$-codes. 
Due to the large number of subcodes we have not enumerated all extensions. So far we have enumerated $2459606$ minimal $[23,7,8]_2$ and $31994$ minimal $[24,8,8]_2$ non-isomorphic codes. 
One example is given by the generator matrix
$$
\begin{pmatrix}
111111111111100010000000\\
000000011111111101000000\\
000111100011101100100000\\
011000100100111100010000\\
001001101101110000001000\\
000010111000011100000100\\
110111100001110000000010\\
010001000011110100000001
\end{pmatrix}
$$
with weight enumerator $1+18x^{8}+30x^{9}+30x^{10}+30x^{11}+22x^{12}+42x^{13}+42x^{14}+26x^{15}+15x^{16}$ and an automorphism group of order $6$. (There is also one example with an 
automorphism group of order $18$.) 
We remark that most of the examples satisfy $w_{\min}=8$, $w_{\max}=17$, and all intermediate weights occur. 
Another example, that is $2$-divisible, is given by the generator matrix
$$
\begin{pmatrix}
111111111111100010000000\\
000000011111111101000000\\
000111100011101100100000\\
001011100101110100010000\\
011101100110110000001000\\
001110111101011100000100\\
001001101100001100000010\\
101100011100100000000001
\end{pmatrix}
$$
and has weight enumerator $1+28x^{8}+60x^{10}+72x^{12}+68x^{14}+27x^{16}$. So far, we found $258$ such non-isomorphic examples. 

\medskip

For dimension $k=9$ we have slightly changed our algorithmic approach. Using the fact that adding a parity bit to a
binary code yields a $2$-divisible (also called even) code, we have enumerated all $2$-divisible minimal $[n,9]_2$-codes 
with $n\le 27$. It turns out that there are exactly $5$ such non-isomorphic codes with length $n=27$ and none with a strictly smaller length. 
If $C$ is a minimal $[n,9]_2$-code that is not even, that adding a parity bit yields an even minimal $[n+1,9]_2$-code. Inverting this operation, we have deleted a column of the above five codes in all possible 
ways and obtained $34$ non-isomorphic $[26,9,9]_2$-codes of which exactly $4$ are minimal, i.e., we have $m(9,2)=26$. 
One example is given by
$$
  \begin{pmatrix}
    11111111110000000100000000\\
    00001111111111100010000000\\
    01110001110011111001000000\\
    00110010010101101000100000\\
    11010010101100111000010000\\
    01110110000010110000001000\\
    01101010110110001000000100\\
    10011100101001011000000010\\
    11001101001100010000000001
  \end{pmatrix}
$$
with weight enumerator $1+32x^{9}+62x^{10}+64x^{11}+84x^{12}+64x^{13}+44x^{14}+64x^{15}+43x^{16}+32x^{17}+22x^{18}$ and an automorphism group of order $16$.

\medskip

For dimension $k=10$  we remark that \cite[Section II.A]{cohen1994intersecting} reports an example verifying $m(10,2)\le 30$. The idea was 
to puncture a $4$-divisible (cyclic) minimal $[31,10,12]_2$ code. In Section~\ref{sec_div} we have determined all $4$-divisible minimal $[31,10,12]_2$ codes. There are exactly two such non-isomorphic 
codes and also two non-isomorphic puncturings with generator matrices
$$
  \begin{pmatrix}
  111111111111110000001000000000\\
  000000111111111110000100000000\\
  001111000011110001110010000000\\
  010111001100110110010001000000\\
  111011010100011000110000100000\\
  010101100000011011110000010000\\
  111010101101011111010000001000\\
  011100000101111100100000000100\\
  111101011001000011010000000010\\
  110110000111000010110000000001
  \end{pmatrix}
  \text{ and }
  \begin{pmatrix}
  111111111111110000001000000000\\
  000000111111111110000100000000\\
  001111000011110001110010000000\\
  010111001100110110010001000000\\
  111011010100011000110000100000\\
  010101100000011011110000010000\\
  100110100001101111010000001000\\
  110100001110011100100000000100\\
  001010001001011110110000000010\\
  011100111010100011010000000001  
  \end{pmatrix}.
$$  
The codes both have an automorphism group of order five and weight enumerator $1+120x^{11}+190x^{12}+272x^{15}+255x^{16}+120x^{19}+66x^{20}$.  
   
\medskip

In order to construct small minimal codes in dimensions $11$ and $12$ we consider a geometric construction. If $M$ is a multiset of points and $Q$ is  
a point in $\operatorname{PG}(v-1q)$, where $v\ge 2$, then we can construct a multiset $M_Q$ by projection trough $Q$, that is the multiset image under 
the map $P\mapsto \langle P, Q\rangle/Q$ setting $M_Q(L/Q) = M(L)-M(Q)$ for every line $L\ge P$ in $\operatorname{PG}(v-1,q)$. We directly verify the 
following properties:
\begin{lemma}   
  \label{lemma_projection}
  Let $M$ be a strong blocking multiset $\PG(k-1,q)$, where $k\ge 2$, and let $M_Q$ arise from $M$ by projection through a point 
  $Q$. Then we have $\# M_Q=\# M-M(Q)$, the span of $M_Q$ has dimension $k-1$, and $M_Q$ is a strong blocking multiset.
\end{lemma}  
By $M'$ we denote the set of points that have positive multiplicity in $M_Q$, so that also $M'$ is a strong blocking (multi-)set in 
$\operatorname{PG}(k-1,q)/Q \cong \operatorname{PG}(k-2,q)$, i.e., we can reduce points with multiplicity larger than one to multiplicity one. 
So, starting from a minimal $[n,k]_q$-code $C$ we consider the corresponding multiset of points $M$, apply projection through a point $Q$, reduce 
point multiplicities to obtain $M'$, and then consider the corresponding minimal $[\# M',k]_q$-code $C'$.
 
As an example we consider the binary code
$$
\begin{pmatrix}
1111110010000\\
0001111101000\\
1110010100100\\
0010101100010\\
0101010100001
\end{pmatrix} 
$$
attaining $m(5,2)=13$. Choosing $Q$ as the first column of the generator matrix gives the code $C'$ with generator matrix
$$
  \begin{pmatrix}
    001111101\\
    001100110\\
    010101100\\
    101010100
  \end{pmatrix}\!\!,
$$
which is a representation of the unique code attaining $m(4,2)=9$, i.e., the union of three disjoint lines. In our examples the lines through column $1$ that contain at least three points (which is the 
maximum for $q=2$ and projective codes) are given by the triples of column indices $(1,2,13)$, $(1,3,12)$, and $(1,9,11)$. Also choosing the point $Q$ as the second column yields a minimal $[9,4]_2$-code, 
while all other columns yield (minimal) codes of larger lengths. For projective binary codes or point sets $M$ in $\operatorname{PG}(k-1,2)$ the geometric description of the cardinality of 
$M'$ equals $\# M-1$ minus the number of full lines through $Q$. I.e., if $Q$ equals the first or the second column, then there are exactly three full lines through $Q$, which is the 
maximum since $m(4,2)\ge 9$. If $Q$ equals the last column then there is unique full line through $Q$ and there are exactly two full lines through $Q$ in all other cases.   

Applying projection to the second non-isomorphic code attaining $m(5,2)=13$ yields minimal $[10,4]_2$- and a minimal $[12,4]_2$-code. Applying projection to the unique minimal $[9,4]_2$-code 
yields the unique minimal $[6,3]_2$-code in all cases. This continues for dimension three and two, as can be easily seen from the geometric description of the extremal point sets. 
Applying projection to the unique minimal $[15,6]_2$-code yields minimal $[13,5]_2$-codes in all cases (which all have automorphism groups of order $48$, i.e.\ are equivalent to second non-isomorphic 
$[13,5]_2$-code). We remark that in \cite[Table I]{sloane1993covering} the example for a minimal $[13,5]_2$-code was described as {\lq\lq}omit coordinates 1,6 from{\rq\rq} the (unique) minimal 
$[15,6]_2$-code. In the same vein a minimal $[29,9]_2$-code was constructed from a minimal $[31,10]_2$-code. We remark that applying projection to the minimal $[26,9]_2$-code 
$$
\begin{pmatrix}
1 1 1 1 1 1 1 1 1 1 1 0 0 0 0 0 0 1 0 0 0 0 0 0 0 0\\
0 0 0 0 0 1 1 1 1 1 1 1 1 1 1 0 0 0 1 0 0 0 0 0 0 0\\
0 0 1 1 1 0 0 0 1 1 1 0 0 1 1 1 1 0 0 1 0 0 0 0 0 0\\
0 1 0 1 1 0 0 1 0 0 1 0 1 0 1 0 1 0 0 0 1 0 0 0 0 0\\
1 1 1 0 0 0 1 1 1 1 0 1 0 0 0 1 1 0 0 0 0 1 0 0 0 0\\
0 0 1 0 1 1 1 1 0 1 0 0 0 1 1 1 0 0 0 0 0 0 1 0 0 0\\
1 0 0 0 1 1 0 0 1 1 1 0 1 1 0 0 1 0 0 0 0 0 0 1 0 0\\
1 0 1 1 0 1 0 1 1 1 0 1 1 1 1 1 1 0 0 0 0 0 0 0 1 0\\
1 1 0 0 1 0 0 1 1 1 0 1 1 0 1 0 0 0 0 0 0 0 0 0 0 1
\end{pmatrix}
$$
gives minimal $[n,8]_2$-codes for $n\in\{24,25\}$. 
This phenomenon also occurs for field sizes larger than $2$. 
    
\medskip

The inversion of the projection transformation gives rise to an integer linear programming formulation to search for minimal codes of small length. Starting with the first 
minimal $[30,10]_2$ code let us find the following minimal $[35,11]_2$ code with generator matrix
$$
  \begin{pmatrix}
  11011110101100100010110010010101000\\
  01000000000011000110110110000011100\\
  00110000000101000110111101011100100\\
  00001000000010011000011100000010111\\
  00000100000101011110111010101110101\\
  00000010000111010110010101100000001\\
  00000001100011010000111011010010001\\
  00000000010001001110010111001110011\\
  00000000001111000000001111000001111\\
  00000000000000111110000000111111111\\
  00000000000000000001111111111111111
  \end{pmatrix},
$$    
weight enumerator $1+19x^{11}+83x^{12}+142x^{13}+118x^{14}+125x^{15}+194x^{16}+296x^{17}+356x^{18}+237x^{19}+141x^{20}+134x^{21}+102x^{22}+67x^{23}+29x^{24}+4x^{25}$ and a trivial automorphism group. 
Applying the approach again yields the following minimal $[40,12]_2$ code with generator matrix 
$$
  \begin{pmatrix}
  1001110000000101010110100110011101010111\\
  0100011001011100100110000000011000111101\\
  0000011111001000110010100000011000110010\\
  0000001001010111001110010000000100111000\\
  0011001101011000111100000000001000001011\\
  0000100011010111101000001000010000001101\\
  0000010011010011011100000100000100000011\\
  0000010111001100111010000010001100001111\\
  0000001111000111100110000001000000000001\\
  0000000000111100011110000000111100000111\\
  0000000000000011111110000000000011111111\\
  0000000000000000000001111111111111111111
  \end{pmatrix},
$$    
weight enumerator $1+21x^{12}+70x^{13}+120x^{14}+173x^{15}+183x^{16}+261x^{17}+408x^{18}+493x^{19}+560x^{20}+521x^{21}+408x^{22}+319x^{23}+240x^{24}+167x^{25}+88x^{26}+39x^{27}
+19x^{28}+5x^{29}$ and a trivial automorphism group.  
We remark that both ILP computations were aborted before finishing.     

\medskip

For larger dimensions the most successful approaches are based on our notion of generalized circulant matrices and corresponding computer searches. We state the obtained examples in
the remaining part of this section. The generator matrices of a $2$-divisible minimal $[29,10]_2$ and a minimal $[35,11]_2$-code are given by
$$
  \begin{pmatrix}
    11111111100000000001000000000\\
    01110000011111111000100000000\\
    11101100011100000100010000000\\
    10010011100011100100001000000\\
    01110011010010010100000100000\\
    11001010001011001100000010000\\
    11011100110110100010000001000\\
    00000101111011010010000000100\\
    11100010110010101110000000010\\
    00101000001110011110000000001
  \end{pmatrix}
  \text{ and }
  \begin{pmatrix}
  010000000000    1011    0011    0011    0010    0010    0010\\
  001000000000    1101    1001    1001    0001    0001    0001\\
  000100000000    1110    1100    1100    1000    1000    1000\\
  000010000000    0101    1011    1100    0101    0010    0100\\
  000001000000    1010    1101    0110    1010    0001    0010\\
  000000100000    0101    1110    0011    0101    1000    0001\\
  000000010000    1010    0111    1001    1010    0100    1000\\
  000000001000    0101    0011    0010    1100    1110    0111\\
  000000000100    1010    1001    0001    0110    0111    1011\\
  000000000010    0101    1100    1000    0011    1011    1101\\
  000000000001    1010    0110    0100    1001    1101    1110\\
  \end{pmatrix}.
$$
The latter code arises from a generalized circulant matrix of a non-minimal $[35,12]_2$-code by removing the first row. 
A minimal $[39,12]_2$-code can be obtained from the following generalized circulant matrix of type $\left(3^3,3^{13},3\right)$:  
$$
  \begin{pmatrix} 
  100\,\,000\,\,000\,\,000\,\,010\,\,010\,\,010\,\,010\,\,001\,\,110\,\,101\,\,110\,\,010\\
  010\,\,000\,\,000\,\,000\,\,001\,\,001\,\,001\,\,001\,\,100\,\,011\,\,110\,\,011\,\,001\\
  001\,\,000\,\,000\,\,000\,\,100\,\,100\,\,100\,\,100\,\,010\,\,101\,\,011\,\,101\,\,100\\[0.5mm]
  000\,\,100\,\,000\,\,000\,\,010\,\,010\,\,001\,\,001\,\,011\,\,011\,\,100\,\,000\,\,101\\
  000\,\,010\,\,000\,\,000\,\,001\,\,001\,\,100\,\,100\,\,101\,\,101\,\,010\,\,000\,\,110\\
  000\,\,001\,\,000\,\,000\,\,100\,\,100\,\,010\,\,010\,\,110\,\,110\,\,001\,\,000\,\,011\\[0.5mm]
  000\,\,000\,\,100\,\,000\,\,010\,\,101\,\,000\,\,110\,\,101\,\,001\,\,101\,\,100\,\,100\\
  000\,\,000\,\,010\,\,000\,\,001\,\,110\,\,000\,\,011\,\,110\,\,100\,\,110\,\,010\,\,010\\
  000\,\,000\,\,001\,\,000\,\,100\,\,011\,\,000\,\,101\,\,011\,\,010\,\,011\,\,001\,\,001\\[0.5mm]
  000\,\,000\,\,000\,\,100\,\,111\,\,011\,\,111\,\,011\,\,010\,\,010\,\,010\,\,011\,\,010\\
  000\,\,000\,\,000\,\,010\,\,111\,\,101\,\,111\,\,101\,\,001\,\,001\,\,001\,\,101\,\,001\\
  000\,\,000\,\,000\,\,001\,\,111\,\,110\,\,111\,\,110\,\,100\,\,100\,\,100\,\,110\,\,100\\
  \end{pmatrix}\!.
$$ 
A minimal $[43,13]_2$-code can be obtained from the following generalized circulant matrix of type $\left(1^1 6^2,1^1 6^7,6\right)$:
$$
  \begin{pmatrix}
  100000\,\,000000\,\,0\,\,101000\,\,101000\,\,111010\,\,111000\,\,100000\\
  010000\,\,000000\,\,0\,\,010100\,\,010100\,\,011101\,\,011100\,\,010000\\
  001000\,\,000000\,\,0\,\,001010\,\,001010\,\,101110\,\,001110\,\,001000\\
  000100\,\,000000\,\,0\,\,000101\,\,000101\,\,010111\,\,000111\,\,000100\\
  000010\,\,000000\,\,0\,\,100010\,\,100010\,\,101011\,\,100011\,\,000010\\
  000001\,\,000000\,\,0\,\,010001\,\,010001\,\,110101\,\,110001\,\,000001\\[0.5mm]
  000000\,\,100000\,\,0\,\,110101\,\,000001\,\,010110\,\,100000\,\,111100\\
  000000\,\,010000\,\,0\,\,111010\,\,100000\,\,001011\,\,010000\,\,011110\\
  000000\,\,001000\,\,0\,\,011101\,\,010000\,\,100101\,\,001000\,\,001111\\
  000000\,\,000100\,\,0\,\,101110\,\,001000\,\,110010\,\,000100\,\,100111\\
  000000\,\,000010\,\,0\,\,010111\,\,000100\,\,011001\,\,000010\,\,110011\\
  000000\,\,000001\,\,0\,\,101011\,\,000010\,\,101100\,\,000001\,\,111001\\[0.5mm]
  000000\,\,000000\,\,1\,\,000000\,\,111111\,\,111111\,\,000000\,\,111111
\end{pmatrix}\!.
$$
The blocks of width or length $1$ might also be described by generalizing the notion of a bordered circulant matrix, see e.g.\ 
\cite[Chapter 16]{macwilliams1977theory}. A minimal $[56,16]_2$-code can be obtained from the following generalized circulant matrix of type $\left(8^2,8^7,8\right)$:
$$
  \begin{pmatrix} 
  10000000\,\,00000000\,\,01100000\,\,01011000\,\,01001000\,\,01100011\,\,01000111\\
  01000000\,\,00000000\,\,00110000\,\,00101100\,\,00100100\,\,10110001\,\,10100011\\
  00100000\,\,00000000\,\,00011000\,\,00010110\,\,00010010\,\,11011000\,\,11010001\\
  00010000\,\,00000000\,\,00001100\,\,00001011\,\,00001001\,\,01101100\,\,11101000\\
  00001000\,\,00000000\,\,00000110\,\,10000101\,\,10000100\,\,00110110\,\,01110100\\
  00000100\,\,00000000\,\,00000011\,\,11000010\,\,01000010\,\,00011011\,\,00111010\\
  00000010\,\,00000000\,\,10000001\,\,01100001\,\,00100001\,\,10001101\,\,00011101\\
  00000001\,\,00000000\,\,11000000\,\,10110000\,\,10010000\,\,11000110\,\,10001110\\[0.5mm]
  00000000\,\,10000000\,\,01110010\,\,11111101\,\,11110010\,\,01101010\,\,11100101\\
  00000000\,\,01000000\,\,00111001\,\,11111110\,\,01111001\,\,00110101\,\,11110010\\
  00000000\,\,00100000\,\,10011100\,\,01111111\,\,10111100\,\,10011010\,\,01111001\\
  00000000\,\,00010000\,\,01001110\,\,10111111\,\,01011110\,\,01001101\,\,10111100\\
  00000000\,\,00001000\,\,00100111\,\,11011111\,\,00101111\,\,10100110\,\,01011110\\
  00000000\,\,00000100\,\,10010011\,\,11101111\,\,10010111\,\,01010011\,\,00101111\\
  00000000\,\,00000010\,\,11001001\,\,11110111\,\,11001011\,\,10101001\,\,10010111\\
  00000000\,\,00000001\,\,11100100\,\,11111011\,\,11100101\,\,11010100\,\,11001011
\end{pmatrix}\!.
$$
For minimal $[38,12]_2$-, 
$[48,14]_2$-, $[52,15]_2$-, 
and $[62,17]_2$-codes we obtained the generator matrices
$$
  \!\!\!\!\!\!\!\!\!\!\!\!\!\!\!\!\!\!\!\!\!\!\!
  \left(\begin{smallmatrix}
  100000000000\,\,111110\,\,111110\,\,110000\,\,100000\,\,10\\
  010000000000\,\,011111\,\,011111\,\,011000\,\,010000\,\,01\\
  001000000000\,\,101111\,\,101111\,\,001100\,\,001000\,\,10\\
  000100000000\,\,110111\,\,110111\,\,000110\,\,000100\,\,01\\
  000010000000\,\,111011\,\,111011\,\,000011\,\,000010\,\,10\\
  000001000000\,\,111101\,\,111101\,\,100001\,\,000001\,\,01\\[0.5mm]
  000000100000\,\,100001\,\,111110\,\,010101\,\,110100\,\,01\\
  000000010000\,\,110000\,\,011111\,\,101010\,\,011010\,\,10\\
  000000001000\,\,011000\,\,101111\,\,010101\,\,001101\,\,01\\
  000000000100\,\,001100\,\,110111\,\,101010\,\,100110\,\,10\\
  000000000010\,\,000110\,\,111011\,\,010101\,\,010011\,\,01\\
  000000000001\,\,000011\,\,111101\,\,101010\,\,101001\,\,10  
  \end{smallmatrix}\right)\!\!,
  \left(\begin{smallmatrix}
  10000000000000\,\,1101100\,\,1011111\,\,0110000\,\,1101000\,\,010000\\
  01000000000000\,\,0110110\,\,1101111\,\,0011000\,\,0110100\,\,101110\\
  00100000000000\,\,0011011\,\,1110111\,\,0001100\,\,0011010\,\,101000\\
  00010000000000\,\,1001101\,\,1111011\,\,0000110\,\,0001101\,\,111100\\
  00001000000000\,\,1100110\,\,1111101\,\,0000011\,\,1000110\,\,111111\\
  00000100000000\,\,0110011\,\,1111110\,\,1000001\,\,0100011\,\,000010\\
  00000010000000\,\,1011001\,\,0111111\,\,1100000\,\,1010001\,\,111011\\[0.5mm]
  00000001000000\,\,1111101\,\,1110000\,\,1101101\,\,1111010\,\,100110\\
  00000000100000\,\,1111110\,\,0111000\,\,1110110\,\,0111101\,\,101001\\
  00000000010000\,\,0111111\,\,0011100\,\,0111011\,\,1011110\,\,011110\\
  00000000001000\,\,1011111\,\,0001110\,\,1011101\,\,0101111\,\,011011\\
  00000000000100\,\,1101111\,\,0000111\,\,1101110\,\,1010111\,\,111110\\
  00000000000010\,\,1110111\,\,1000011\,\,0110111\,\,1101011\,\,011111\\
  00000000000001\,\,1111011\,\,1100001\,\,1011011\,\,1110101\,\,111000
  \end{smallmatrix}\right)\!\!,
$$
$$
  \left(\begin{smallmatrix}
  100000000000000\,\,1110110\,\,1011000\,\,1110000\,\,1010000\,\,1100000\,\,11\\
  010000000000000\,\,0111011\,\,0101100\,\,0111000\,\,0101000\,\,0110000\,\,11\\
  001000000000000\,\,1011101\,\,0010110\,\,0011100\,\,0010100\,\,0011000\,\,11\\
  000100000000000\,\,1101110\,\,0001011\,\,0001110\,\,0001010\,\,0001100\,\,11\\
  000010000000000\,\,0110111\,\,1000101\,\,0000111\,\,0000101\,\,0000110\,\,11\\
  000001000000000\,\,1011011\,\,1100010\,\,1000011\,\,1000010\,\,0000011\,\,11\\
  000000100000000\,\,1101101\,\,0110001\,\,1100001\,\,0100001\,\,1000001\,\,11\\[0.5mm]
  000000010000000\,\,1100000\,\,1010010\,\,1001010\,\,1100110\,\,1100101\,\,01\\
  000000001000000\,\,0110000\,\,0101001\,\,0100101\,\,0110011\,\,1110010\,\,01\\
  000000000100000\,\,0011000\,\,1010100\,\,1010010\,\,1011001\,\,0111001\,\,01\\
  000000000010000\,\,0001100\,\,0101010\,\,0101001\,\,1101100\,\,1011100\,\,01\\
  000000000001000\,\,0000110\,\,0010101\,\,1010100\,\,0110110\,\,0101110\,\,01\\
  000000000000100\,\,0000011\,\,1001010\,\,0101010\,\,0011011\,\,0010111\,\,01\\
  000000000000010\,\,1000001\,\,0100101\,\,0010101\,\,1001101\,\,1001011\,\,01\\[0.5mm]
  000000000000001\,\,1111111\,\,1111111\,\,0000000\,\,1111111\,\,0000000\,\,00\\ 
  \end{smallmatrix}\right)\!,\text{ and}
$$
$$
  \left(\begin{smallmatrix}
  10000000000000000\,\,01011001110011000\,\,01010110100110100\,\,01001000100\\
  01000000000000000\,\,00101100111001100\,\,00101011010011010\,\,11100000010\\
  00100000000000000\,\,00010110011100110\,\,00010101101001101\,\,00001000010\\
  00010000000000000\,\,00001011001110011\,\,10001010110100110\,\,00000000010\\
  00001000000000000\,\,10000101100111001\,\,01000101011010011\,\,01111110110\\
  00000100000000000\,\,11000010110011100\,\,10100010101101001\,\,01111110011\\
  00000010000000000\,\,01100001011001110\,\,11010001010110100\,\,01110011100\\
  00000001000000000\,\,00110000101100111\,\,01101000101011010\,\,01001111100\\
  00000000100000000\,\,10011000010110011\,\,00110100010101101\,\,01010010110\\
  00000000010000000\,\,11001100001011001\,\,10011010001010110\,\,11100100110\\
  00000000001000000\,\,11100110000101100\,\,01001101000101011\,\,10001001011\\
  00000000000100000\,\,01110011000010110\,\,10100110100010101\,\,01000011010\\
  00000000000010000\,\,00111001100001011\,\,11010011010001010\,\,11000110010\\
  00000000000001000\,\,10011100110000101\,\,01101001101000101\,\,11010111101\\
  00000000000000100\,\,11001110011000010\,\,10110100110100010\,\,01110001110\\
  00000000000000010\,\,01100111001100001\,\,01011010011010001\,\,11111011000\\
  00000000000000001\,\,10110011100110000\,\,10101101001101000\,\,10000100111
  \end{smallmatrix}\right)\!,
$$
respectively. Here we did not decompose the preceding unit matrix into blocks and only the submatrix 
without the last block of columns is obtained as a generalized circulant matrix. The columns from the 
last blocks are carefully chosen step by step in order to turn the linear code into a minimal one. To this 
end we present some measure of distance to a minimal linear code in the subsequent subsection.

Further examples of minimal $[35,11]_2$- and $[48,14]_2$-codes are given by
$$
  \left(\begin{smallmatrix}
  10000000000\,\,111\,\,111\,\,111\,\,111\,\,111\,\,000\,\,000\,\,011\\[0.5mm]
  01000000000\,\,110\,\,100\,\,100\,\,100\,\,100\,\,100\,\,100\,\,101\\
  00100000000\,\,011\,\,010\,\,010\,\,010\,\,010\,\,010\,\,010\,\,101\\
  00010000000\,\,101\,\,001\,\,001\,\,001\,\,001\,\,001\,\,001\,\,101\\[0.5mm]
  00001000000\,\,010\,\,111\,\,011\,\,101\,\,010\,\,110\,\,010\,\,100\\
  00000100000\,\,001\,\,111\,\,101\,\,110\,\,001\,\,011\,\,001\,\,100\\
  00000010000\,\,100\,\,111\,\,110\,\,011\,\,100\,\,101\,\,100\,\,100\\[0.5mm]
  00000001000\,\,001\,\,110\,\,101\,\,111\,\,010\,\,000\,\,110\,\,110\\
  00000000100\,\,100\,\,011\,\,110\,\,111\,\,001\,\,000\,\,011\,\,110\\
  00000000010\,\,010\,\,101\,\,011\,\,111\,\,100\,\,000\,\,101\,\,110\\[0.5mm]
  00000000001\,\,000\,\,000\,\,111\,\,111\,\,111\,\,111\,\,000\,\,110
  \end{smallmatrix}\right)\!\text{ and }
  \left(\begin{smallmatrix}
  10000000000000\,\,111111\,\,111111\,\,111111\,\,111111\,\,000000\,\,0011\\[0.5mm]
  01000000000000\,\,111110\,\,111010\,\,110100\,\,111000\,\,101000\,\,0101\\
  00100000000000\,\,011111\,\,011101\,\,011010\,\,011100\,\,010100\,\,0101\\
  00010000000000\,\,101111\,\,101110\,\,001101\,\,001110\,\,001010\,\,0101\\
  00001000000000\,\,110111\,\,010111\,\,100110\,\,000111\,\,000101\,\,0101\\
  00000100000000\,\,111011\,\,101011\,\,010011\,\,100011\,\,100010\,\,0101\\
  00000010000000\,\,111101\,\,110101\,\,101001\,\,110001\,\,010001\,\,0101\\[0.5mm]
  00000001000000\,\,000111\,\,100110\,\,000010\,\,101000\,\,110111\,\,1010\\
  00000000100000\,\,100011\,\,010011\,\,000001\,\,010100\,\,111011\,\,1010\\
  00000000010000\,\,110001\,\,101001\,\,100000\,\,001010\,\,111101\,\,1010\\
  00000000001000\,\,111000\,\,110100\,\,010000\,\,000101\,\,111110\,\,1010\\
  00000000000100\,\,011100\,\,011010\,\,001000\,\,100010\,\,011111\,\,1010\\
  00000000000010\,\,001110\,\,001101\,\,000100\,\,010001\,\,101111\,\,1010\\[0.5mm]
  00000000000001\,\,111111\,\,111111\,\,000000\,\,000000\,\,111111\,\,0111
  \end{smallmatrix}\right)\!\!.
$$

\subsection{Acute sets}
\label{subsec_acute} 
Around 1950 Paul Erd\H{o}s conjectured that given more than $2^d$ points in $\mathbb{R}^d$ there are three of
them determining an obtuse angle, i.e.\ an angle strictly greater than $\pi/2$. This conjecture is indeed true, 
see \cite{danzer1962zwei},\cite[Chapter 17]{proofsfromthebook}, and an example is given by a $d$-dimensional 
hypercube which contains many angles of degree $\pi/2$. A set of points in $\mathbb{R}^d$ is acute, if any three points 
from this set form an acute angle, i.e.\ strictly less than $\pi/2$. Such so-called \emph{acute sets} can have 
exponential size \cite{gerencser2019acute} and the maximum possible sizes of acute sets in $\{0,1\}^d$ up to dimension 
$d=10$ are stated in A089676 of the {\lq\lq}The On-Line Encyclopedia of Integer Sequences{\rq\rq} (OEIS). We say that 
a set $S\subseteq \{0,1\}^d$ is linear if it is linearly closed when interpreted over $\F_2$.
\begin{lemma}(Cf.~\cite{randriambololona2017metric})
  Let $C$ be an $[n,k]_2$-code. The codewords of $C$ form an acute set iff $C$ is minimal. 
\end{lemma}
\begin{proof}
  We associate $C$ with the set $S\subseteq\{0,1\}^k\subset\mathbb{R}^k$ of codewords of $C$ and only use $C$ in the following. 
  Since the points are are subset of the $k$-dimensional unit cube the angle between any triple of points $z,b,c$ is at most $\pi/2$. 
  W.l.o.g.\ we assume that $z$ is the zero vector. So, the angle between $b$, $z$, $c$ at $z=0$ is $\pi/2$ iff the scalar product 
  vanishes, i.e.\ $\sum_{i=1}^n b_ic_i=0$.

  Now consider $a,b\in\mathbb{F}_2^n$ with $\supp(b)\subset \supp(a)$ and set $c=a+b$, so that $\supp(b)\cap\supp(c)=\emptyset$ 
  and $\sum_{i=1}^n b_ic_i=0$. So, if $C$ is acute it is also minimal. For the other direction we observe that 
  $\sum_{i=1}^n b_ic_i=0$ implies $\supp(b)\cap\supp(c)=\emptyset$.  
\end{proof}
Up to dimension $d=4$ the maximum size of an acute set in $\{0,1\}^d$ is indeed attained by a binary linear code. Up to isomorphism there are 
exactly five acute sets in $\{0,1\}^9$ with maximum cardinality $16$ -- only one of them is linear. 
In $\{0,1\}^{10}$ the number of non-isomorphic acute sets of maximum possible cardinality $17$ is $655$, clearly none of them linear. 
For dimension $11$ we performed a partial search finding $17$ non-isomorphic acute sets
of size $23$ and two of size $24$.   
Additionally we have checked that all acute sets in $\{0,1\}^9$ with cardinality at least $10$ have 
extensions to 11-dimensional acute sets with cardinality at most $20$ and all acute sets in $\{0,1\}^{10}$ 
with cardinality $17$ have extensions to 11-dimensional acute sets with cardinality at most $19$. There are 
more than 60\,000  acute sets with cardinality $16$ in $\{0,1\}^{10}$. Using an 
integer linear programming formulation we have checked that the cardinality of all 11-dimensional extensions 
of acute sets in $\{0,1\}^{9}$ with cardinality $8$ or $9$ is upper bounded by $28$.  
Thus, the maximum cardinality of an acute set in $\{0,1\}^{11}$ is upper bounded by $28$.

If we build up a linear code column by column or by adding full column blocks of circulant matrices, 
then our intermediate codes are not minimal and we need some kind of measurement for the \text{distance} 
to a minimal code in our heuristic searches. To this end we use the number of angles with value $\pi/2$.

For additional relations of minimal codes to other structures we refer e.g.\ to \cite{scotti2024recent}.
    
\section*{Acknowledgments}
The first author would like to thank amateur programmer Olga Briginets for her help in writing computer programs. 
The second author thanks Gianira Alfarano, Anurag Bishnoi, Jozefien D'haeseleer, Dion Gijswijt, Alessandro Neri, 
Sven Polak, and Martin Scotti for many helpful remarks on an earlier version of this paper, 
which originally started to investigate so-called trifferent codes, see \cite{trifferent_kurz}.


\newcommand{\etalchar}[1]{$^{#1}$}

\end{document}